\documentclass[leqno,12pt]{article}
\usepackage{amssymb}
\usepackage{amsmath}
\usepackage{amsfonts}
\usepackage{graphicx}%
\providecommand{\U}[1]{\protect \rule{.1in}{.1in}}
\newtheorem{theorem}{Theorem}

\newtheorem{corollary}[theorem]{Corollary}

\newtheorem{proposition}[theorem]{Proposition}

\newenvironment{proof}[1][Proof]{\noindent \textbf{#1.} }{\  \rule{0.5em}{0.5em}}
\begin{document}

\title{\textbf{Charakterisierungen spezieller windschiefer Regelfl\"{a}chen \\ durch die
Normalkr\"{u}mmung ausgezeichneter Fl\"{a}chenkurven\bigskip}}
\author{\textbf{Stylianos S. Stamatakis\strut \bigskip}\\
Abteilung f\"{u}r Mathematik \\ der Aristoteles Universit\"{a}t
Thessaloniki\\
GR-54124 Thessaloniki, Griechenland\\
e-mail: stamata@math.auth.gr}
\date{}
\maketitle

\begin{abstract}
\noindent We consider ruled surfaces in the three-dimensional Euclidean space
and some geometrically distinguished families of curves on them whose normal
curvature has a concrete form.
The aim of this paper is to find and classify
all ruled surfaces with the mentioned property
\smallskip \vspace{0.1cm}

\end{abstract}

\section{Einleitung}

Geometrisch ausgezeichnete Kurvenscharen auf windschiefen Regelf\"{a}chen des
euklidischen Raumes $%
\mathbb{R}
^{3}$ sind von verschiedenen Autoren vielfach untersucht worden. Eine Reihe
von Resultaten ergeben sich, wenn man fordert, dass die betrachtete
Kurvenschar eine zus\"{a}tzliche Eigenschaft besitzt. Die vorliegende Arbeit
liefert einen Beitrag zu diesem Themenkreis. Wir betrachten spezielle
Kurvenscharen auf einer windschiefen Regelfl\"{a}che und nehmen an, dass die
Normalkr\"{u}mmung l\"{a}ngs jeder dieser Kurven eine bestimmte Gestalt
besitzt. Unser Ziel ist die Klassifikation der Regelfl\"{a}chen mit der
geforderten Eigenschaft.
\smallskip \vspace{0.2cm}

Im euklidischen Raum $%
\mathbb{R}
^{3}$ sei $\Phi$ eine regul\"{a}re und von torsalen Erzeugenden freie
Regelfl\"{a}che mit der Striktionslinie $\boldsymbol{s}=\boldsymbol{s}(u)$ und
dem Erzeugendeneinheitsvektor $\boldsymbol{e}(u)$. $\Phi$ sei \"{u}ber dem
Definitionsgebiet $G:=I\times%
\mathbb{R}
$ ($I\subset%
\mathbb{R}
$ offenes Intervall) von der Klasse $C^{3}$. Eine Parameterdarstellung von
$\Phi$ lautet dann%
\begin{equation}
\boldsymbol{x}(u,v)=\boldsymbol{s}(u)+v\, \boldsymbol{e}(u),\;u\in I,\;v\in%
\mathbb{R}
. \label{1}%
\end{equation}
Wegen der Annahme gilt%
\begin{equation}
\langle \boldsymbol{s}\,%
\acute{}%
\,(u),\, \boldsymbol{e}\,%
\acute{}%
\,(u)\rangle=0\; \forall \;u\in I,
\end{equation}
wobei Strich Ableitung nach $u$ und $\langle$ , $\rangle$ das
Standard-Skalarprodukt im Raume $%
\mathbb{R}
^{3}$ bedeuten. Den Parameter $u$ w\"{a}hlen wir so, dass%
\begin{equation}
\left \vert \boldsymbol{e}\,%
\acute{}%
\,(u)\right \vert =1\; \forall \;u\in I
\end{equation}
gilt.
\\[2mm]
L\"{a}ngs der Striktionslinie betrachten wir das begleitende Dreibein
\{$\boldsymbol{e}(u)$, $\boldsymbol{n}(u)$, $\boldsymbol{z}(u)$\}, wobei
$\boldsymbol{n}(u):=\boldsymbol{e}\,%
\acute{}%
\,(u)$ der \textit{Zentralnormalenvektor} und $\boldsymbol{z}%
(u):=\boldsymbol{e}(u)\times \boldsymbol{n}(u)$ der
\textit{Zentraltangentenvektor} von $\Phi$ ist. Es gelten folgende
Ableitungsgleichungen :%
\begin{equation}
\boldsymbol{e}\,%
\acute{}%
\,(u)=\boldsymbol{n}(u),\quad \boldsymbol{n}\,%
\acute{}%
\,(u)=-\boldsymbol{e}(u)+k(u)\, \boldsymbol{z}(u),\quad \boldsymbol{z}\,%
\acute{}%
\,(u)=-k(u)\, \boldsymbol{n}(u),
\end{equation}
wobei%
\begin{equation}
k(u)=(\boldsymbol{e}(u),\, \boldsymbol{e}\,%
\acute{}%
\,(u),\, \boldsymbol{e}\,%
\acute{}%
\, \,%
\acute{}%
\,(u))
\end{equation}
die \textit{konische Kr\"{u}mmung} von $\Phi$ bedeutet. F\"{u}r den
\textit{Drall} $\delta(u)$ und die \textit{Striktion} $\sigma(u)$ von $\Phi$
haben wir%
\begin{equation}
\delta(u)=(\boldsymbol{e}(u),\, \boldsymbol{e}\,%
\acute{}%
\,(u),\, \boldsymbol{s}\,%
\acute{}%
\,(u)),\quad \sigma(u):=\sphericalangle(\boldsymbol{e}(u),\, \boldsymbol{s}\,%
\acute{}%
\,(u))
\end{equation}
($-\frac{\pi}{2}$ $<\sigma \leq \frac{\pi}{2},\operatorname*{sign}%
\sigma=\operatorname*{sign}\delta$). Die Funktionen $k(u)$, $\delta(u)$ und
$\sigma(u)$ bilden bekanntlich ein vollst\"{a}ndiges Invariantensystem von
$\Phi$ (vgl. \cite{Hoschek}, S.19).\medskip \newline Die Koordinatenfunktionen
$g_{ij}$ bzw. $h_{ij}$ des ersten bzw. zweiten Fundamentaltensors
bez\"{u}glich der (lokalen) Koordinaten $u^{1}:=u$, $u^{2}:=v$ lauten%
\begin{equation}
\left \{
\begin{array}
[c]{c}%
{\normalsize (g}_{ij}{\normalsize )=}\left(
\begin{array}
[c]{cc}%
v^{2}+\delta^{2}\left(  \lambda^{2}+1\right)  & \delta \, \lambda \\
\delta \, \lambda & 1
\end{array}
\right) \\
{\small (h}_{ij}{\small )=}\frac{1}{w}\left(
\begin{array}
[c]{cc}%
-\left[  k\,v^{2}{\small +}\delta \,%
\acute{}%
v{\small +}\delta^{2}\left(  k-\lambda \right)  \right]  & \delta \\
\delta & 0
\end{array}
\right)
\end{array}
\right.  ,
\end{equation}
mit $w:=\sqrt{v^{2}+\delta^{2}}$.
\\[2mm]
F\"{u}r die Gau\ss sche Kr\"{u}mmung $K$ und
die mittlere Kr\"{u}mmung $H$ von $\Phi$ gelten die Beziehungen
\begin{equation}
K=\frac{-\delta^{2}}{w^{4}},\quad H=-\frac{k\,v^{2}+\delta \,%
\acute{}%
\,v+\delta^{2}\left(  k+\lambda \right)  }{2\,w^{3}},
\end{equation}
wobei $\lambda:=\cot \sigma$ gesetzt wurde.\smallskip

Eine windschiefe Regelfl\"{a}che $\Phi$, deren s\"{a}mtliche Schmiegquadriken
Drehhyperboloide sind, bezeichnet man bekanntlich als
\textit{Edlinger-Fl\"{a}che} \cite{Edlinger}, \cite{Hoschek}. Hinreichende und
notwendige Bedingungen daf\"{u}r sind die Beziehungen (\cite{Brauner}, S.103)%
\begin{equation}
\delta \,%
\acute{}%
\,=k\, \lambda+1=0.
\end{equation}
Es handelt sich um eine konstant gedrallte Regelfl\"{a}che mit der
Striktionslinie als Kr\"{u}mmugslinie. Die \textit{Kurven konstanten
Striktionsabstandes}, d.h. die Kurven $v=konst.$, sind in diesem Fall
Kr\"{u}mmungslinien von $\Phi$. Die andere Kr\"{u}mmungslinienschar wird durch%
\[
\lbrack k^{2}\,v^{2}+\delta^{2}(k^{2}+1)]\,du-\delta \,k\,dv=0
\]
bestimmt. Die zugeh\"{o}rigen Normalkr\"{u}mmungen der beiden
Kr\"{u}mmungslinien (Hauptkr\"{u}mmungen) lauten%
\begin{equation}
k_{1}=-k(u)\,w^{-1},\quad k_{2}=\frac{\delta^{2}(u)}{k(u)}w^{-3}.
\end{equation}

\section{Der Fall der Hauptkr\"{u}mmungen}

Im folgenden betrachten wir \textit{ausschlie\ss lich windschiefe
Regelfl\"{a}chen des Raumes} $%
\mathbb{R}
^{3}$ \textit{mit der Parameterdarstellung} (\ref{1}), \textit{die die
Voraussetzungen }(2) und (3) \textit{erf\"{u}llen.}
\smallskip \vspace{0.2cm}
Ausgehend von (10) stellen wir zun\"{a}chst die Aufgabe, \textit{alle}
\textit{Regelfl\"{a}chen zu bestimmen, deren eine Hauptkr\"{u}mmung die
Gestalt}%
\begin{equation}
k_{i}=f(u)\,w^{n},\;n\in%
\mathbb{Z}
,\;f(u)\in C^{0}(I),\;i=1\; \text{oder}\;2,
\end{equation}
\textit{besitzt}. Es ist offenbar $f(u)\neq0\; \forall \;u\in I$, denn $\Phi$
ist windschief.\smallskip

F\"{u}r die Normalkr\"{u}mmung in Richtung $du:dv$ findet man unter Beachtung
von (7)%
\begin{equation}
k_{N}=\frac{1}{w}\cdot \frac{-\left[  kv^{2}+\delta \,%
\acute{}%
\,v+\delta^{2}\left(  k-\lambda \right)  \right]  \,du^{2}+2\, \delta
\,du\,dv}{\left[  v^{2}+\delta^{2}\left(  \lambda^{2}+1\right)  \right]
\,du^{2}+2\, \delta \, \lambda \,du\,dv+dv^{2}}\text{,}%
\end{equation}
woraus mit (11)%
\begin{align*}
&  [f\,w^{n+1\,}[v^{2}+\delta^{2}(\lambda^{2}+1)]+k\,v^{2}+\delta \,%
\acute{}%
\,v+\delta^{2}(k-\lambda)]\,du^{2}\\
&  +2\, \delta \,(f\, \lambda \,w^{n+1}-1)\,du\,dv+f\,w^{n+1\,}dv^{2}=0
\end{align*}
folgt. Diese in $du:dv$ quadratische Gleichung besitzt genau dann
\textit{eine} L\"{o}sung, wenn ihre Diskriminante verschwindet:%
\begin{equation}
f^{2}\,w^{2n+4}+f\,[k\,v^{2}+\delta \,%
\acute{}%
\,v+\delta^{2}(k+\lambda)]\,w^{n+1}-\delta^{2}=0\quad \forall \;u\in I,\;v\in%
\mathbb{R}
.
\end{equation}
Es sei zun\"{a}chst $n=0$. Dann erhalten wir aus (13)%
\[
f^{4}(v^{2}+\delta^{2})^{4}-2\,f^{2}\delta^{2}(v^{2}+\delta^{2})^{2}%
+\delta^{4}-f^{2\,}[k\,v^{2}+\delta \,%
\acute{}%
\,v+\delta^{2}(k+\lambda)]^{2}(v^{2}+\delta^{2})=0.
\]
Auf der linken Seite steht ein Polynom achten Grades in $v$, das f\"{u}r jedes
$u\in I$ und unendlich viele Werte $v\in%
\mathbb{R}
$ zu erf\"{u}llen ist. Durch Koeffizientenvergleich mit dem Nullpolynom ergibt
sich $f=0$, also ein Widerspruch.
\\[2mm]
Wir unterscheiden nun folgende F\"{a}lle:
\\[2mm] \textit{Fall I}: Es sei $n$ ungerade. Die
Beziehung (13) erh\"{a}lt die Gestalt%
\begin{align}
Q(v):  &  =f\,[k\,v^{2}+\delta \,%
\acute{}%
\,v+\delta^{2}(k+\lambda)](v^{2}+\delta^{2})^{\frac{n+1}{2}}\nonumber \\
&  +f^{2}(v^{2}+\delta^{2})^{n+2}-\delta^{2}=0\quad \forall \;u\in I,\;v\in%
\mathbb{R}
.
\end{align}
Ist $n\geq1$, so liefert das Verschwinden des Koeffizienten des h\"{o}chsten
Grades des Polynoms $Q(v)$ $f=0$, was nicht m\"{o}glich ist.\smallskip \newline
Es sei $n=-1$. Dann wird (14)%
\[
Q(v)=f^{2}(v^{2}+\delta^{2})+f\,[k\,v^{2}+\delta \,%
\acute{}%
\,v+\delta^{2}(k+\lambda)]-\delta^{2}=0.
\]
Aus dem Verschwinden der Koeffizienten des Polynoms $Q(v)$ erhalten wir%
\[
f=-k,\quad \delta \,%
\acute{}%
\,=k\, \lambda+1=0,
\]
d.h. $\Phi$ muss eine Edlinger-Fl\"{a}che sein.\smallskip \newline Es sei
$n=-3$ und $k_{1}$ die Hauptkr\"{u}mmung, die die Gestalt (11) besitzt, d.h.
\[
k_{1}=f(u)\,w^{-3}.
\]
Dann erhalten wir aus (8) f\"{u}r die andere Hauptkr\"{u}mmung%

\[
k_{2}=f^{\ast}(u)\,w^{-1}\text{ mit \ }f^{\ast}(u):=\frac{-\delta^{2}%
(u)}{f(u)},
\]
d.h. eine Hauptkr\"{u}mmung von $\Phi$ hat die Gestalt (11), wobei $n=-1$ ist.
Wie wir im vorangehenden Fall gesehen haben, ist $\Phi$ wieder eine
Edlinger-Fl\"{a}che.\smallskip \newline Der Fall $n\leq-5$ f\"{u}hrt auf einen
Widerspruch, wie man leicht feststellen kann.\medskip \newline \textit{Fall II}:
Es sei $n$ gerade. F\"{u}r $n=-2$ folgt aus (13)%
\[
Q(v):=f^{2}[k\,v^{2}+\delta \,%
\acute{}%
\,v+\delta^{2}(k+\lambda)]^{2}-(f^{2}-\delta^{2})^{2}(v^{2}+\delta^{2}%
)=0\quad \forall \;u\in I,\;v\in%
\mathbb{R}
.
\]
Das Verschwinden des Koeffizienten $f^{2}k^{2}$ der Potenz h\"{o}chsten Grades
des Polynoms $Q(v)$ liefert $k=0$. Aus dem Verschwinden der \"{u}brigen
Koeffizienten folgt%
\[
f^{2}\, \delta \,%
\acute{}%
\,^{2}-(f^{2}-\delta^{2})^{2}=2f^{2}\, \delta^{2}\, \delta \,%
\acute{}%
\, \lambda=f^{2\,}\delta^{4\,}\lambda^{2}-\delta^{2}(f^{2}-\delta^{2})^{2}=0,
\]
woraus wir
\[
\delta \,\acute{} \,=\lambda=0 \quad \text{und} \quad f=\pm \, \delta
\]
erhalten. Es handelt sich also in diesem
Fall um eine Wendelfl\"{a}che.\smallskip \newline Die F\"{a}lle $n\geq2$ und
$n\leq4$ f\"{u}hren auf Widerspr\"{u}che, wie man leicht zeigen
kann.\smallskip \newline Wir fassen unsere Ergebnisse zusammen:

\begin{proposition}
Es sei $\Phi \subset%
\mathbb{R}
^{3}$ eine windschiefe $C^{3}$-Regelfl\"{a}che, deren eine Hauptkr\"{u}mmung
die Gestalt (11) besitzt. Dann gilt eine der folgenden Eigenschaften:\newline%
(a)$\quad n=-1$, $f(u)=-k(u)$ und $\Phi$ ist eine Edlinger-Fl\"{a}che.\newline%
(b)$\quad n=-2$, $f(u)=\pm \, \delta(u)$ und $\Phi$ ist eine
Wendelfl\"{a}che.\newline(c)$\quad n=-3$, $f(u)=\delta^{2}(u)\,k^{-1}(u)$ und
$\Phi$ ist eine Edlinger-Fl\"{a}che.
\end{proposition}

Aus diesem Satz folgt unmittelbar folgendes

\begin{corollary}
Es sei $\Phi \subset%
\mathbb{R}
^{3}$ eine windschiefe $C^{3}$-Regelfl\"{a}che, zwischen deren
Hauptkr\"{u}mmungen die Beziehung%
\begin{equation}
\delta^{2\,}k%
\genfrac{}{}{0pt}{1}{3}{1}%
+k^{4}\,k_{2}=0
\end{equation}
bestehe. Dann ist $\Phi$ eine Edlinger-Fl\"{a}che.
\end{corollary}
\begin{proof}
Auf Grund von (8) und (15) folgt
\[
k\genfrac{}{}{0pt}{1}{4}{1}=k^{4}\,w^{-4}
\]
und daraus
\[
k_{1}=\pm \,k\,w^{-1}.
\]
Wegen Satz 1 ist aber
\[
k_{1}=-k\,w^{-1}.
\]
Somit besitzt die Hauptkr\"{u}mmung $k_{1}$ die Gestalt
(11), wobei $n=-1$ ist. Daher ist $\Phi$ eine Edlinger-Fl\"{a}che.
\end{proof}

\section{Der Fall der Normalkr\"{u}mmung}

Dar\"{u}er hinaus betrachten wir weitere geometrisch ausgezeichnete
Kurvenscharen auf $\Phi$, l\"{a}ngs deren die Normalkr\"{u}mmung die Gestalt%
\begin{equation}
k_{N}=f(u)w^{n},\;n\in%
\mathbb{Z}
,\;f(u)\in C^{0}(I)
\end{equation}
besitzt. Unser Ziel ist die Bestimmung dieser Regelfl\"{a}chen.\smallskip
\medskip

\noindent \textbf{3.1. }Es sei $S_{1}$ \textit{die Schar der Kurven konstanten
Striktionsabstandes}. Nach (12) lautet die Normalkr\"{u}mmung l\"{a}ngs einer
Kurve von $S_{1}$%
\begin{equation}
k_{N}=\frac{1}{w}\cdot \frac{-k\,v^{2}-\delta \,%
\acute{}%
\,v+\delta^{2}\left(  k-\lambda \right)  }{v^{2}+\delta^{2}\left(  \lambda
^{2}+1\right)  }%
\end{equation}
und hat genau dann die Gestalt (16), wenn%
\begin{equation}
f\,w^{n+1\,}[v^{2}+\delta^{2}(\lambda^{2}+1)]+k\,v^{2}+\delta \,%
\acute{}%
\,v+\delta^{2}(k-\lambda)=0\text{.}%
\end{equation}
Offensichtlich ist genau dann $f=0$, wenn%
\begin{equation}
k\,v^{2}+\delta \,%
\acute{}%
\,v+\delta^{2}(k-\lambda)=0
\end{equation}
identisch in $v$ gilt, d.h. genau dann, wenn
\[
k=\delta \,\acute{}\,=\lambda=0,
\]
also wenn $\Phi$ eine Wendelfl\"{a}che ist.\smallskip \newline
Es sei nun $f\neq0$. F\"{u}r $n=-1$ folgt aus (18)%
\[
Q(v):=f[v^{2}+\delta^{2}(\lambda^{2}+1)]+k\,v^{2}+\delta \,%
\acute{}%
\,v+\delta^{2}(k-\lambda)=0.
\]
Das Verschwinden der Koeffizienten des Polynoms $Q(v)$ liefert%
\[
f=-k,\quad \delta \,%
\acute{}%
\,=0,\quad \lambda(k\lambda+1)=0.
\]
Somit ist die Regelfl\"{a}che $\Phi$ entweder ein konstant gedralltes Orthoid
($\delta \,%
\acute{}%
\,=\lambda=0$) oder eine Edlinger-Fl\"{a}che ($\delta \,%
\acute{}%
\,=k\lambda+1=0$).\smallskip \newline F\"{u}r $n>-1$ folgt aus (18)%
\[
Q(v):=f^{2}(v^{2}+\delta^{2})^{n+1}[v^{2}+\delta^{2}(\lambda^{2}%
+1)]^{2}-[k\,v^{2}+\delta \,%
\acute{}%
\,v+\delta^{2}(k-\lambda)]^{2}=0.
\]
Aus dem Verschwinden des Koeffizienten des h\"{o}chsten Grades des Polynoms
$Q(v)$ folgt $f=0$. Dies ist aber nicht m\"{o}glich.\smallskip \newline F\"{u}r
$n<-1$ folgt aus (18)%
\[
Q(v):=(v^{2}+\delta^{2})^{-n-1}[k\,v^{2}+\delta \,%
\acute{}%
\,v+\delta^{2}(k-\lambda)]^{2}-f^{2}[v^{2}+\delta^{2}(\lambda^{2}+1)]^{2}=0.
\]
Das Verschwinden des Koeffizienten des h\"{o}chsten Grades des Polynoms $Q(v)$
liefert $k=0$ und demnach%
\begin{equation}
Q(v)=(v^{2}+\delta^{2})^{-n-1}(\delta \,%
\acute{}%
\,v-\delta^{2}\, \lambda)^{2}-f^{2}[v^{2}+\delta^{2}(\lambda^{2}+1)]^{2}=0.
\end{equation}
Ist $n=-2$, so erh\"{a}lt $Q(v)$ die Gestalt%
\begin{equation}
Q(v)=(v^{2}+\delta^{2})(\delta \,%
\acute{}%
\,v-\delta^{2}\lambda)^{2}-f^{2}[v^{2}+\delta^{2}(\lambda^{2}+1)]^{2}=0.
\end{equation}
Aus dem Verschwinden der Koeffizienten von $Q(v)$ folgt $f=0$, also ein
Widerspruch.\newline F\"{u}r $n<-2$ liefert das Verschwinden des Koeffizienten
des h\"{o}chsten Grades des Polynoms $Q(v)$ in (20) $\delta \,%
\acute{}%
\,=0$ und daher%
\begin{equation}
Q(v)=\delta^{4}\lambda^{2}(v^{2}+\delta^{2})^{-n-1}-f^{2}[v^{2}+\delta
^{2}(\lambda^{2}+1)]^{2}=0.
\end{equation}
F\"{u}r $n=-3$ ist%
\[
Q(v)=\delta^{4}\lambda^{2}(v^{2}+\delta^{2})^{2}-f^{2}[v^{2}+\delta
^{2}(\lambda^{2}+1)]^{2}=0.
\]
Aus dem Verschwinden der Koeffizienten des Polynoms $Q(v)$ folgt wieder $f=0$,
also ein Widerspruch.\smallskip \newline F\"{u}r $n<-3$ erhalten wir aus (22)
\[
\lambda=f=0,
\]
was ebenfalls unm\"{o}glich ist.\smallskip \newline Somit ist
folgender Satz bewiesen:

\begin{proposition}
Die Normalkr\"{u}mmung l\"{a}ngs der Kurven konstanten Striktionsabstandes
einer windschiefen $C^{3}$-Regelfl\"{a}che $\Phi \subset%
\mathbb{R}
^{3}$ besitze die Gestalt (16). Dann gilt eine der folgenden
Eigenschaften:\newline(a)\quad$f=0$ und $\Phi$ ist eine
Wendelfl\"{a}che.\newline(b)\quad$n=-1,f(u)=-k(u)$ und $\Phi$ ist ein konstant
gedralltes Orthoid oder eine Edlinger-Fl\"{a}che.\medskip
\end{proposition}

\noindent \textbf{3.2.} Es sei $S^{2}$ \textit{die Schar der
Orthogonaltrajektorien der Schar} $S^{1}$. Sie ist durch%
\[
\lbrack v^{2}+\delta^{2}(\lambda^{2}+1)]\,du+\delta \, \lambda \,dv=0
\]
bestimmt. Nach (12) lautet die zugeh\"{o}rige Normalkr\"{u}mmung%
\begin{equation}
k_{N}=\frac{1}{w^{3}}\cdot \frac{-\delta^{2}\lambda \, \left[  \left(  k\,
\lambda+2\right)  v^{2}+\delta \,%
\acute{}%
\, \lambda \,v+\delta^{2}\left(  \lambda^{2}+k\, \lambda+2\right)  \right]
}{v^{2}+\delta^{2}\left(  \lambda^{2}+1\right)  }%
\end{equation}
und hat genau dann die Gestalt (16), wenn gilt%
\begin{equation}
f\,w^{n+3}[v^{2}+\delta^{2}(\lambda^{2}+1)]+\delta^{2}\lambda \lbrack(k\,
\lambda+2)v^{2}+\delta \,%
\acute{}%
\, \lambda \,v+\delta^{2}(\lambda^{2}+k\, \lambda+2)]=0.
\end{equation}
Es ist $f=0$ genau dann, wenn%
\[
\lambda \,[(k\, \lambda+2)v^{2}+\delta \,%
\acute{}%
\, \lambda \,v+\delta^{2}(\lambda^{2}+k\, \lambda+2)]=0
\]
identisch in $v$ gilt. W\"{a}re $\lambda \neq0$, so w\"{u}rden wir folgern%
\[
k\, \lambda+2=\delta \,%
\acute{}%
\, \lambda=\delta^{2}(\lambda^{2}+k\, \lambda+2)=0,
\]
was nicht m\"{o}glich ist. Somit ist $f=0$ genau dann, wenn \thinspace
$\lambda=0$, d.h. wenn $\Phi$ ein Orthoid ist.\smallskip \newline Es sei nun
$f\lambda \neq0.$ F\"{u}r $n=-3$ folgt aus (24)%
\begin{equation}
Q(v):=f\,[v^{2}+\delta^{2}(\lambda^{2}+1)]+\delta^{2}\lambda \,[(k\,
\lambda+2)v^{2}+\delta \,%
\acute{}%
\, \lambda \,v+\delta^{2}(\lambda^{2}+k\, \lambda+2)]=0.
\end{equation}
Aus dem Verschwinden der Koeffizienten des Polynoms Q(v) erhalten wir%
\[
f=-\delta^{2}\lambda \,(k\, \lambda+2),\quad \delta \,%
\acute{}%
\,=0,\quad k\, \lambda+1=0,
\]
somit haben wir \thinspace$f=-\delta^{2}\, \lambda$ und die Regelfl\"{a}che
$\Phi$ ist eine Edlinger-Fl\"{a}che.\smallskip \newline Man \"{u}berzeugt sich
leicht, dass die F\"{a}lle $n>-3$ und $n<-3$ auf Widerspr\"{u}che f\"{u}hren.
Somit gilt der

\begin{proposition}
Die Normalkr\"{u}mmung l\"{a}ngs der Orthogonaltrajektorien der Kurven
konstanter Striktionsabstandes einer windschiefen $C^{3}$-Regelfl\"{a}che
$\Phi \subset%
\mathbb{R}
^{3}$ besitze die Gestalt (16). Dann gilt eine der folgenden
Eigenschaften:\newline(a)$\quad f=0$ und $\Phi$ ist ein Orthoid.\newline%
(b)$\quad n=-3$, $f(u)=\delta^{2}(u)\,k^{-1}(u)$ und $\Phi$ ist eine
Edlinger-Fl\"{a}che.\medskip
\end{proposition}

\noindent \textbf{3.3.} Es sei $S^{3}$ \textit{die Schar der
Orthogonaltrajektorien der Erzeugenden}, die durch%
\[
\delta \, \lambda \,du+dv=0
\]
bestimmt ist. Nach (12) lautet die zugeh\"{o}rige Normalkr\"{u}mmung%
\begin{equation}
k_{N}=\frac{-k\,v^{2}-\delta \,%
\acute{}%
\,v-\delta^{2}\left(  k+\lambda \right)  }{w^{3}}%
\end{equation}
und hat genau dann die Gestalt (16), wenn gilt%
\begin{equation}
f\,w^{n+3}+k\,v^{2}+\delta \,%
\acute{}%
\,v+\delta^{2}(k+\lambda)=0.
\end{equation}
Es ist $f=0$ genau dann, wenn (19) identisch in $v$ gilt, d.h. genau dann,
wenn
\[
k=\lambda=\delta \,%
\acute{}%
\,=0,
\]
also wenn $\Phi$ eine Wendelfl\"{a}che ist.\smallskip \newline Es sei
nun $f\neq0$. F\"{u}r $n=-3$ folgt aus (27)%
\begin{equation}
Q(v):=k\,v^{2}+\delta \,%
\acute{}%
\,v+\delta^{2}(k+\lambda)+f=0.
\end{equation}
Das Verschwinden der Koeffizienten des Polynoms $Q(v)$ liefert%

\[
f=-\delta^{2}\, \lambda,\quad k=0,\quad \delta \,%
\acute{}%
\,=0.
\]
Somit ist die Regelfl\"{a}che $\Phi$ ein konstant gedralltes Konoid.\smallskip
\newline F\"{u}r $n=-2$ folgt aus (27)%
\[
Q(v):=f^{2}(v^{2}+\delta^{2})-[k\,v^{2}+\delta \,%
\acute{}%
\,v+\delta^{2}(k+\lambda)]^{2}=0.
\]
Aus dem Verschwinden der Koeffizienten des Polynoms $Q(v)$ folgt%
\[
k=0,\quad f^{2}=\delta \,%
\acute{}%
\,^{2},\quad \delta \,%
\acute{}%
\, \lambda=0,\quad f^{2}=\delta^{2}\lambda^{2}=0,
\]
also $f=0$, d.h. ein Widerspruch.\smallskip \newline F\"{u}r $n=-1$ folgt aus
(27)%
\[
Q(v):=f(v^{2}+\delta^{2})+k\,v^{2}+\delta \,%
\acute{}%
\,v+\delta^{2}(k+\lambda)=0.
\]
Aus dem Verschwinden der Koeffizienten des Polynoms $Q(v)$ folgt
$f=-k,\delta \,%
\acute{}%
\,=0,\lambda=0$. Somit ist $\Phi$ ein konstant gedralltes Orthoid.\smallskip
\newline Die F\"{a}lle $n\geq0$ und $n\leq-4$ f\"{u}hren auf Widerspr\"{u}che,
wie man leicht zeigen kann. Daher gilt:

\begin{proposition}
Die Normalkr\"{u}mmung l\"{a}ngs der Orthogonaltrajektorien der Erzeugenden
einer windschiefen $C^{3}$-Regelfl\"{a}che $\Phi \in%
\mathbb{R}
^{3}$ besitze die Gestalt (16). Dann gilt eine der folgenden
Eigenschaften:\newline(a)$\quad f=0$ und $\Phi$ ist eine
Wendelfl\"{a}che.\newline(b)$\quad n=-3$, $f(u)=-\delta^{2}(u)\, \lambda(u)$
und $\Phi$ ist ein konstant gedralltes Konoid.\newline(c)$\quad n=-1$,
$f(u)=-k(u)$ und $\Phi$ ist ein konstant gedralltes Orthoid.\medskip
\end{proposition}

\noindent \textbf{3.4.} Es sei $S^{4}$ \textit{die Schar der Kurven konstanter
Gau\ss scher Kr\"{u}mmung}, d.h. die Schar der Kurven l\"{a}ngs denen die
Gau\ss sche Kr\"{u}mmung konstant ist \cite{Sachs}. Sie ist durch%
\[
\delta \,%
\acute{}%
\,(\delta^{2}-v^{2})\,du+2\delta \,v\,dv=0
\]
bestimmt. Nach (12) lautet die zugeh\"{o}rige Normalkr\"{u}mmung
\begin{equation}
k_{N}=\frac{-1}{w}\cdot \frac{4\delta^{2}\,v\left[  k\,v^{3}+\delta^{2}\left(
k-\lambda \right)  v+\delta^{2}\, \delta \,%
\acute{}%
\, \right]  }{A}%
\end{equation}
wobei
\[
A=\left(  4\delta^{2}+\delta \,%
\acute{}%
\,^{2}\right)  v^{4}+4\delta^{2\,}\delta \,%
\acute{}%
\lambda \,v^{3}+2\delta^{2}\left[  2\delta^{2}\left(  \lambda^{2}+1\right)
-\delta \,%
\acute{}%
\,^{2}\right]  v^{2}-4\delta^{4\,}\delta \,%
\acute{}%
\, \lambda \,v+\delta^{4\,}\delta \,%
\acute{}%
\,^{2}%
\]
gesetzt wurde, und hat genau dann die Gestalt (16), wenn gilt:%
\begin{align}
&  \text{ }f\,w^{n+1}[(4\delta^{2}+\delta \,%
\acute{}%
\,^{2})v^{4}+4\delta^{2}\, \delta \,%
\acute{}%
\, \lambda \,v^{3}+2\delta^{2}[2\delta^{2}(\lambda^{2}+1)-\delta \,%
\acute{}%
\,^{2}]v^{2}\nonumber \\
&  -4\delta^{4\,}\delta \,%
\acute{}%
\, \lambda \,v+\delta^{4\,}\delta \,%
\acute{}%
\,^{2}]+4\delta^{2}v\,[k\,v^{3}+\delta^{2}(k-\lambda)v+\delta^{2}\delta \,%
\acute{}%
\,]=0.
\end{align}
Es ist $f=0$ genau dann, wenn%
\[
k\,v^{3}+\delta^{2}(k-\lambda)v+\delta^{2\,}\delta \,%
\acute{}%
\,=0
\]
identisch in $v$ erf\"{u}llt ist, d.h. genau dann, wenn
\[
k=\lambda=\delta \,%
\acute{}%
\,=0,
\]
also wenn $\Phi$ eine Wendelfl\"{a}che ist.\smallskip \newline Es sei im
Folgenden $f\neq0$. F\"{u}r $n=-1$ folgt aus (30)%
\begin{align}
Q(v)  &  :=f\,[(4\delta^{2}+\delta \,%
\acute{}%
\,^{2})v^{4}+4\delta^{2}\, \delta \,%
\acute{}%
\, \lambda \,v^{3}+2\delta^{2\,}[2\delta^{2}(\lambda^{2}+1)-\delta \,%
\acute{}%
\,^{2}]v^{2}\nonumber \\
&  -4\delta^{4}\, \delta \,%
\acute{}%
\, \lambda \,v+\delta^{4}\delta \,%
\acute{}%
\,^{2}]+4\delta^{2}\,v[k\,v^{3}+\delta^{2}(k-\lambda)v+\delta^{2\,}\delta \,%
\acute{}%
\,]=0.
\end{align}
Die Koeffizienten
\begin{align*}
a_{4}  &  :=f\,(4\delta^{2}+\delta \,%
\acute{}%
\,^{2})+4\delta^{2}k,\quad a_{3}:=4f\, \delta^{2}\, \delta \,%
\acute{}%
\, \lambda,\\
a_{2}  &  :=2f\, \delta^{2}[2\delta^{2}(\lambda^{2}+1)-\delta \,%
\acute{}%
\,^{2}]+4\delta^{4}(k-\lambda),\\
a_{1}  &  :=-4f\, \delta^{4}\delta \,%
\acute{}%
\, \lambda+4\delta^{4}\, \delta \,%
\acute{}%
\,,\quad a_{0}:=f\, \delta^{4\,}\delta \,%
\acute{}%
\,^{2}%
\end{align*}
des Polynoms $Q(v)$ m\"{u}ssen verschwinden. Aus $a_{0}=0$ folgt
\thinspace$\delta \,%
\acute{}%
\,=0$ und sodann aus $a_{2}=a_{4}=0$, dass
\[
f=-k, \quad \lambda\,(k\, \lambda+1)=0
\]
gilt. Folglich ist die Regelfl\"{a}che $\Phi$ entweder ein konstant gedralltes
Orthoid ($\delta \,%
\acute{}%
\,=\lambda=0$) oder eine Edlinger-Fl\"{a}che ($\delta \,%
\acute{}%
\,=k\, \lambda+1=0$).\smallskip \newline Die F\"{a}lle $n>-1$ und $n<-1$
f\"{u}hren auf Widerspr\"{u}che. Wir haben somit folgenden Satz bewiesen:

\begin{proposition}
Die Normalkr\"{u}mmung l\"{a}ngs der Kurven konstanten Gau\ss scher
Kr\"{u}mmung einer windschiefen $C^{3}$-Regelfl\"{a}che $\Phi \in%
\mathbb{R}
^{3}$ besitze die Gestalt (16). Dann gilt eine der folgenden
Eigenschaften:\newline(a)$\quad f=0$ und $\Phi$ ist eine
Wendelfl\"{a}che.\newline(b)$\quad n=-1$, $f(u)=-k(u)$ und $\Phi$ ist ein
konstant gedralltes Orthoid oder eine Edlinger-Fl\"{a}che.
\end{proposition}

Folgende Tabelle gibt eine \"{U}bersicht \"{u}ber die obigen
Resultate.\bigskip \newline%
\begin{tabular}
[c]{|c|c|c|c|}\hline
$%
\begin{tabular}
[c]{c}%
{\small Normalkr\"{u}mmung der }\\
{\small Gestalt }$k_{N}=f\,w^{n}$ {\small l\"{a}ngs}%
\end{tabular}
$ & ${\small f}$ & $n$ & {\small Art der Regelfl\"{a}che}\\ \hline \hline%
\begin{tabular}
[c]{c}%
{\small einer der }\\
{\small Kr\"{u}mmungslinien}%
\end{tabular}
&
\begin{tabular}
[c]{c}%
${\small -k}$\\
${\small \pm \delta}$\\
${\small \delta}^{2}{\small k}^{-1}$%
\end{tabular}
&
\begin{tabular}
[c]{c}%
${\small -1}$\\
${\small -2}$\\
${\small -3}$%
\end{tabular}
&
\begin{tabular}
[c]{c}%
$\cdot$\ {\small Edlinger-Fl\"{a}che}\\
$\cdot$\ {\small Wendelfl\"{a}che}\\
$\cdot$ {\small Edlinger-Fl\"{a}che}%
\end{tabular}
\\ \hline%
\begin{tabular}
[c]{c}%
{\small der Kurven konstanten}\\
{\small Striktionsabstandes}%
\end{tabular}
&
\begin{tabular}
[c]{c}%
${\small 0}$\\
${\small -k}$%
\end{tabular}
&
\begin{tabular}
[c]{c}%
{\small -}\\
${\small -1}$%
\end{tabular}
&
\begin{tabular}
[c]{c}%
$\cdot$ {\small Wendelfl\"{a}che}\\%
\begin{tabular}
[c]{c}%
$\cdot$\ {\small konstant gedralltes Orthoid}\\
{\small oder Edlinger-Fl\"{a}che}%
\end{tabular}
\end{tabular}
\\ \hline%
\begin{tabular}
[c]{c}%
{\small der Orthogonaltrajektorien}\\
{\small der Kurven konstanten}\\
{\small Striktionsabstandes}%
\end{tabular}
&
\begin{tabular}
[c]{c}%
${\small 0}$\\
${\small \delta}^{2}{\small k}^{-1}$%
\end{tabular}
&
\begin{tabular}
[c]{c}%
{\small -}\\
${\small -3}$%
\end{tabular}
&
\begin{tabular}
[c]{c}%
$\cdot$\ {\small Orthoid}\\
$\cdot$\ {\small Edlinger-Fl\"{a}che}%
\end{tabular}
\\ \hline%
\begin{tabular}
[c]{c}%
{\small der Orthogonaltrajektorien}\\
{\small der Erzeugenden}%
\end{tabular}
&
\begin{tabular}
[c]{c}%
${\small 0}$\\
${\small -k}$\\
${\small -\delta}^{2}{\small \lambda}$%
\end{tabular}
&
\begin{tabular}
[c]{c}%
{\small -}\\
${\small -1}$\\
${\small -3}$%
\end{tabular}
&
\begin{tabular}
[c]{c}%
$\cdot$\ {\small Wendelfl\"{a}che}\\
$\cdot$\ {\small konstant gedralltes Orthoid}\\
$\cdot$\ {\small konstant gedralltes Konoid}%
\end{tabular}
\\ \hline%
\begin{tabular}
[c]{c}%
{\small der Kurven }\\
{\small konstanter}\\
{\small Gau\ss scher Kr\"{u}mmung}%
\end{tabular}
&
\begin{tabular}
[c]{c}%
${\small 0}$\\
${\small -k}$%
\end{tabular}
&
\begin{tabular}
[c]{c}%
{\small -}\\
${\small -1}$%
\end{tabular}
&
\begin{tabular}
[c]{c}%
$\cdot$\ {\small Wendelfl\"{a}che}\\%
\begin{tabular}
[c]{c}%
$\cdot$ {\small konstant gedralltes Orthoid}\\
{\small oder Edlinger-Fl\"{a}che}%
\end{tabular}
\end{tabular}
\\ \hline
\end{tabular}
\newline 

\noindent

\begin{thebibliography}{9}                                                                                                %
\bibitem {Brauner}H. Brauner: \"{U}ber Strahlfl\"{a}chen von konstantem Drall.
Monatsh. Math. \textbf{63} (1959), 101-111.
\bibitem {Edlinger}R. Edlinger: \"{U}ber Regelfl\"{a}chen, deren s\"{a}mtliche
oskulierenden Hyperboloide Drehhyperboloide sind. S.-B. Akad. Wiss. Wien \textbf{132}
(1923), 243-351.
\bibitem {Hoschek}J. Hoschek: Liniengeometrie. Bibliographisches Institut,
Z\"{u}rich 1971.
\bibitem {Sachs}H. Sachs: Einige Kennzeichnungen der Edlinger-Fl\"{a}chen.
Monatsh. Math. \textbf{77} (1973), 241-250.
\end{thebibliography}
\end{document}